\documentclass[preprint, 12pt, english]{elsarticle}
\usepackage{amsmath}
\usepackage{latexsym, amssymb}
\usepackage{amsthm}
\usepackage{txfonts}

\newtheorem{thm}{Theorem}[section] 

\newtheorem{lem}[thm]{Lemma}

\newtheorem{rem}[thm]{Remark}

\newcommand\operA[2]{{\if!#2!\operatorname{#1}\else{\operatorname{#1}_{#2}^{\phantom{I}}}\fi}} 
\newcommand\set[1]{\{#1\}}

\newcommand\charac[1]{\mathrm{char}\left(#1\right)}
\newcommand{\Span}[1]{\mathrm{Span}\left\{#1\right\}}
\newcommand{\Sp}{\text{Sp}}

\newcommand\Cref[1]{{Corollary~\ref{#1}}}%
\newcommand{\Trace}[1][]{\if!#1!\operatorname{Tr}\else{\operatorname{Tr}_{#1}^{\phantom{I}}}\fi} 

\long\def\forget#1\forgotten{{}} %

\def\({\left(}
\def\){\right)}

\newif\iffurther
\furtherfalse

\newif\ifXY 
\XYtrue     
\ifXY

\input xy
\input xyidioms.tex
\usepackage{xy}
\xyoption{all} %
\fi 

\usepackage{babel}

\journal{Journal of Algebra and its Applications}

\begin{document}

\begin{frontmatter}

\title{Chain Equivalences for Symplectic Bases, Quadratic Forms and Tensor Products of Quaternion Algebras}

\author{Adam Chapman\corref{ch}}
\ead{adam1chapman@yahoo.com}
\cortext[ch]{The author is supported by Wallonie-Bruxelles International.}
\address{ICTEAM Institute, Universit\'{e} Catholique de Louvain, B-1348 Louvain-la-Neuve, Belgium.}

\begin{abstract}
We present a set of generators for the symplectic group which is different from the well-known set of transvections, from which the chain equivalence for quadratic forms in characteristic 2 is an immediate result.
Based on the chain equivalences for quadratic forms, both in characteristic 2 and not 2, we provide chain equivalences for tensor products of quaternion algebras over fields with no nontrivial 3-fold Pfister forms. The chain equivalence for biquaternion algebras in characteristic 2 is also obtained in this process, without any assumption on the base-field.
\end{abstract}


\end{frontmatter}

\section{Introduction}

A chain equivalence is a theorem that provides basic steps with which one can obtain from one presentation of an object all the other presentations.
A theorem of this kind is a useful tool for deciding whether two given objects are isomorphic or not.
If one wants for example to define an invariant of a certain object, then in order to prove that it is independent of the choice of presentation, one just has to show that it is invariant under the basic steps the chain equivalence provides.

There is Witt's well-known chain equivalence for nondegenerate symmetric bilinear forms, which in turn provides a chain equivalence for quadratic forms over fields of characteristic not $2$. (See \cite{EKM} for further details.)
In \cite{Revoy1977QF}, Revoy provided a chain equivalence for quadratic forms in characteristic 2.

In \cite{Albert} (for characteristic 2) and \cite{Merkurjev1} (otherwise) it was proven that every central simple algebra of exponent $2$ is Brauer equivalent to some tensor product of quaternion algebras, which added to the significance of these objects.

The chain equivalence for one quaternion algebra has been known for quite a long time.
Lately, two papers have been written on the chain equivalence for biquaternion algebras over fields of characteristic not $2$, \cite{Siv} and \cite{ChapVish2}.

The symplectic group, denoted by $\Sp_{2 n}(F)$, is the group of automorphisms of a symplectic vector space preserving the symplectic form. It is known to be generated by transvections, i.e. the automorphisms fixing hyperplanes. (See \cite{Dieudonne} for further details.)

\section{Preliminaries}

Let $F$ be a field. We use the following expressions to describe quadratic forms:
\begin{itemize}
\item The expression $\langle a_1,\dots,a_t \rangle$ denotes the diagonal quadratic form $q(u_1,\dots,u_t)=a_1 u_1^2+\dots+a_t u_t^2$.
\item In case of characteristic 2, $[a,b]$ denotes the quadratic form $q(u_1,u_2)=a u_1^2+u_1 u_2+b u_2^2$.
\item The direct sum $q_1(u_1,\dots,u_n) \perp q_2(v_1,\dots,v_m)$ denotes the quadratic form $f(u_1,\dots,u_n,v_1,\dots,v_m)=q_1(u_1,\dots,u_n)+q_2(v_1,\dots,v_m)$.
\item The tensor product $\langle a_1,\dots,a_t \rangle \otimes q$ denotes $a_1 q \perp \dots \perp a_t q$.
\end{itemize}

There is a unique 2-dimensional nondegenerate isotropic quadratic form up to isometry, known as the hyperbolic plane and usually denoted by $\varmathbb{H}$.
In characteristic not $2$, $\varmathbb{H}=\langle 1,-1 \rangle$, and in characteristic $2$, $\varmathbb{H}=[0,1]$. We write $n \varmathbb{H}$ for the direct sum of $n$ copies of $\varmathbb{H}$.
Two nondegenerate quadratic forms $q_1$ and $q_2$ are said to be Witt equivalent if $q_1 \simeq q_2 \perp n \varmathbb{H}$ or $q_2 \simeq q_1 \perp n \varmathbb{H}$ for some $n$. In particular, if the forms are of the same dimension then they are Witt equivalent if and only if they are isometric.
The group $I_q F=I_q^1 F$ is the additive group of Witt equivalence classes of even-dimensional nondegenerate quadratic forms over $F$ with $\perp$ as the group operation.
The subgroup $I_q^n F$ is defined recursively to be the subgroup generated by the forms $\langle 1,-a \rangle \otimes q$ where $a \in F^\times$ and $q \in I_q^{n-1} F$.
Following the traditional abuse of notation, we say that a quadratic form $q$ belongs to $I_q^n$ if its Witt equivalence class belongs to that group.

If $\charac{F} \neq 2$ then every quadratic form $q \in I_q F$ is isometric to some diagonal form $\langle a_1,\dots,a_{2 n} \rangle$ where $a_i \neq 0$ for all $i \in \set{1,\dots,2 n}$. The discriminant of $q$ is the class of $(-1)^n \cdot a_1 \cdot \ldots \cdot a_{2 n}$ in $F^\times / (F^\times)^2$. The discriminant is $1$ if and only if $q \in I_q^2 F$. Every quaternion algebra in this case is isomorphic to $(a,b)=F[x,y : x^2=a, y^2=b, x y=-y x]$ for some $a,b \in F^\times$.

If $\charac{F}=2$ then every quadratic form $q \in I_q F$ is isometric to a form $[a_1,b_1] \perp \dots \perp [a_n,b_n]$ for some $a_1,\dots,a_n,b_1,\dots,b_n \in F$. We call the form $[a_1,b_1] \perp \dots \perp [a_n,b_n]$ ``standard". The discriminant of $q$ (also known as the ``Arf invariant") is the class of $a_1 b_1+\dots+a_n b_n$ in $F/\wp(F)$ where $\wp(F)=\set{a^2+a : a \in F}$. The discriminant is zero if and only if $q \in I_q^2 F$. Every quaternion algebra in this case is isomorphic to $[a,b)=F[x,y : x^2+x=a, y^2=b, x y+y x=y]$ for some $a \in F$ and $b \in F^\times$.

For any quadratic form $q$ in $n$ variables, the Clifford algebra $C(q)$ is defined to be $F[x_1,\dots,x_n : (u_1 x_1+\dots+u_n x_n)^2=q(u_1,\dots,u_n)]$.
If $q \in I_q F$ then $C(q)$ is isomorphic to a tensor product of quaternion algebras over $F$.
There is an epimorphism from $I_q^2 F$ to $Br_2(F)$ taking each quadratic form to its Clifford algebra.
The kernel of this homomorphism is known to be $I_q^3 F$, which means that there is an isomorphism between $I_q^2 F/I_q^3 F$ and $Br_2(F)$ (See \cite{Merkurjev1} for characteristic not 2 and \cite{Sah1972} for characteristic 2).

\section{Symplectic Bases}\label{altform}

Let $V$ be a $2 n$-dimensional symplectic vector space over a field $F$ of arbitrary characteristic for some $n \in \mathbb{N}$, and let $B$ be the symplectic (i.e. nondegenerate alternating bilinear) form on $V$.
A \emph{symplectic basis} of $(V,B)$ is an ordered set $(x_1,y_1|x_2,y_2|\dots|x_n,y_n)$ of $n$ mutually orthogonal symplectic pairs of vectors in $V$.

Let $S$ be the set of automorphisms in $\Sp_{2 n}(F)$ acting on the symplectic basis $(x_1,y_1|\dots|x_n,y_n)$ as either
\begin{enumerate}
\item $(x_k,y_k|x_{k+1},y_{k+1}) \rightarrow (x_k+y_{k+1},y_k|x_{k+1}+y_k,y_{k+1})$
\end{enumerate}
for some $k \in \set{1,\dots,n-1}$, or one of the following:
\begin{enumerate}\setcounter{enumi}{1}
\item $(x_k,y_k) \rightarrow (\beta x_k,\beta^{-1} y_k)$,
\item $(x_k,y_k) \rightarrow (x_k+\beta y_k,y_k)$,
\item $(x_k,y_k) \rightarrow (x_k,y_k+\beta x_k)$,
\end{enumerate}
for some $k \in \set{1,\dots,n}$ and $\beta \in F^\times$.

We write $C \approx C'$ for any two symplectic bases satisfying $s(C)=C'$ for some $s \in S$.
Let $\sim$ be the equivalence relation induced by $\approx$.

We are going to prove that $C \sim C'$ for any two symplectic bases of $(V,B)$, which is equivalent to proving that $S$ generates $\Sp_n(F)$.

\begin{rem}\label{neq1}
If $n=1$ then $\Sp_2(F)$ is known to be generated by automorphisms of Types 2,3 and 4.
\end{rem}

\begin{rem}\label{other}
The following equivalences are easy exercises:
$$(x_1,y_1|x_2,y_2) \sim (x_1+x_2,y_1|x_2,y_2-y_1) \sim (x_1,y_1|x_2,y_2-y_1) \sim (x_1,y_1+y_2|x_2-x_1,y_2).$$
\end{rem}

\begin{lem}\label{permute}
For every $\sigma \in S_n$, $(x_1,y_1|\dots|x_n,y_n) \sim (x_{\sigma(1)},y_{\sigma(1)}|\dots|x_{\sigma(n)},y_{\sigma(n)})$.
\end{lem}

\begin{proof}
It suffices to consider swapping two consecutive pairs.
\begin{eqnarray*}
(x_1,y_1|x_2,y_2) &\sim& (x_1+x_2,y_1|x_2,y_2-y_1) \sim (x_1+x_2,y_2|-x_1,y_2-y_1)\\ &\sim& (x_2,y_2|-x_1,-y_1) \approx (x_2,y_2|x_1,y_1).
\end{eqnarray*}
\end{proof}

\begin{lem}\label{completion}
For any symplectic basis $C$ and $0 \neq w \in V$, $w$ can be completed to a symplectic basis $C'$ such that $C \sim C'$.
\end{lem}

\begin{proof}
The case $n = 1$ is clear. Let $H_i = \Span{x_i,y_i}$, so
$V = H_1 \perp \dots \perp H_n$. By permuting and applying the case $n = 1$ to the
$H_i$-s, we may without loss of generality assume $w = x_1 + \dots + x_m$ with $1 \leq m \leq n$. We now use induction on $m$. If $m = 1$, we are done. If $m > 1$, let $x'= x_2+\dots+x_m$. By induction there exists a basis $(x_1,y_1|x',y'|\dots) \sim C$ and
$C \sim (x_1,y_1|-y',x'|\dots) \approx (x_1 + x',y_1|-y' + y_1,x'|\dots)$ as desired.
\end{proof}

\begin{thm}\label{sympchain}
For any two symplectic bases of $(V,B)$, $C \sim C'$.
\end{thm}

\begin{proof}
We use induction on $n$, the case $n=1$ being Remark \ref{neq1}.
From Lemma \ref{completion}, we are reduced to showing that if $C = (x_1,y_1|x_2,y_2|\dots)$ and $C' = (x_1,v_1|u_2,v_2|\dots)$
are two symplectic bases, then $C \sim C'$. As in that lemma, we may reduce to the case
where $v_1 = r + s$ with $r \in H_1$ and $s \in H_2$. If $s=0$ we are done (case $n=1$),
so we may assume $s=u_2$ (by applying the case $n = 1$ on $H_2$), and clearly $r \neq 0$. But then
$B(x_1,v_1) = B(x_1,r) = 1$, so $(x_1,r)$ is a symplectic basis for $H_1$, and therefore $(x_1,r) \sim (x_1,y_1)$ (case $n = 1$). Also, $C' \sim (-v_1,x_1|-v_2, u_2|\dots) \sim (-v_1 + u_2, x_1|\dots) \sim
(-r,x_1|\dots) \sim (x_1,r|\dots) \sim (x_1,y_1|\dots)$ and we are reduced to showing that
if $C = (x_1,y_1|x_2,y_2|\dots)$ and $C' = (x_1,y_1|u_2,v_2|\dots)$ are two symplectic bases,
then $C \sim C'$. But then $(x_2,y_2|\dots|x_n,y_n)$ and $(u_2,v_2|\dots|u_n,v_n)$ are symplectic
bases of $H_1^\perp$, and we are done by induction.
\end{proof}

The following chain equivalence for quadratic forms in characteristic 2 is an immediate result of Theorem \ref{sympchain}:
\begin{thm}\label{qfchain}(\cite[Proposition 3]{Revoy1977QF})
Over a field of characteristic 2, every two isometric standard quadratic forms are connected by a finite number of steps, where each step is either
\begin{enumerate}
\item[\textbf{A}] $[a_k,b_k] \perp [a_{k+1},b_{k+1}] \rightarrow [a_k+b_{k+1},b_k] \perp [a_{k+1}+b_k,b_{k+1}]$
\end{enumerate}
for some $k \in \set{1,\dots,n-1}$, or one of the following types:
\begin{enumerate}
\item[\textbf{B}] $[a_k,b_k] \rightarrow [\beta^2 a_k,\beta^{-2} b_k]$,
\item[\textbf{C}] $[a_k,b_k] \rightarrow [a_k+\beta^2 b_k+\beta,b_k]$,
\item[\textbf{D}] $[a_k,b_k] \rightarrow [a_k,b_k+\beta^2 a_k+\beta]$,
\end{enumerate}
for some $k \in \set{1,\dots,n}$ and $\beta \in F^\times$.
\end{thm}

\begin{rem}\label{nonzero}
If $[a_1,b_1] \perp \dots \perp [a_n,b_n] \simeq [a_1',b_1'] \perp \dots \perp [a_n',b_n']$ such that $b_i,b_i' \neq 0$ for all $1 \leq i \leq n$, then they are connected by a chain of steps as described in Theorem \ref{qfchain} such that the second coefficient in each summand is nonzero.
\end{rem}

\begin{proof}
Let $[a_1,b_1] \perp \dots \perp [a_n,b_n]=q_0,q_1,\dots,q_m=[a_1',b_1'] \perp \dots \perp [a_n',b_n']$ be a chain as described in Theorem \ref{qfchain}.
Assume that the second coefficient in the first summand in each of the forms $q_k,q_{k+1},\dots,q_l$ is zero for some $1 \leq k \leq l \leq m-1$.
Assume that this interval is maximal, i.e. that the second coefficients in the first summands of $q_{k-1}$ and $q_{l+1}$ are nonzero.
The step $q_{k-1} \rightarrow q_k$ is clearly of Type \textbf{D}, which means that the first summand in $q_{k-1}$ is hyperbolic and equal to $[a,\beta^2 a+\beta]$ for some $a,\beta \in F^\times$. Similarly, the first summand in $q_{l+1}$ is hyperbolic and equal to $[a',\beta'^2 a'+\beta']$ for some $a',\beta' \in F^\times$.
In this interval, every step that involves the first summand changes only the first summand. [It is enough to observe that it holds for a step of Type \textbf{A}, because the other types involve only one summand to begin with.]
Therefore the steps in the interval are divided into two sets - those that change only the first summand and those that do not change the first summand at all.
The steps that change only the first summand eventually change $[a,\beta^2 a+\beta]$ to $[a',\beta'^2 a'+\beta']$, and therefore they can be replaced by the following sequence of steps (of Types \textbf{C}, \textbf{D} and \textbf{C} respectively)
$[a,\beta^2 a+\beta] \rightarrow [0,\beta^2 a+\beta] \rightarrow [0,\beta'^2 a'+\beta'] \rightarrow [a',\beta'^2 a'+\beta']$. We leave the other steps as they are.
By repeating this process for any such interval, we obtain a chain where the second entry in the first coefficient in each of the forms is nonzero.
We can then proceed to the second summand, and the desired chain is obtained inductively.
\end{proof}

\section{Tensor Products of Quaternion Algebras in Characteristic $2$}\label{tensorchar2}

Let $F$ be a field of characteristic 2.

\begin{thm}
If $I_q^3 F=0$ or $n=2$ then every two isomorphic tensor products of $n$ quaternion algebras are connected by a finite number of steps, where each step is either
\begin{enumerate}
\item $[c_k,d_k) \otimes [c_{k+1},d_{k+1}) \rightarrow [c_k+\beta^2 d_k d_{k+1},d_k) \otimes [c_{k+1}+\beta^2 d_k d_{k+1},d_{k+1})$
\end{enumerate}
for some $k \in \set{1,\dots,n-1}$ and $\beta \in F$, or one of the following types:
\begin{enumerate}\setcounter{enumi}{1}
\item $[c_1,d_1) \otimes \dots \otimes [c_n,d_n) \rightarrow [c_1,(1+\beta+\beta^2 \sum_{r=1}^n c_r) d_1) \otimes \dots \otimes [c_n,(1+\beta+\beta^2 \sum_{r=1}^n c_r) d_n)$,
\item $[c_k,d_k) \rightarrow [c_k+\alpha^2 d_k+\beta^2+\beta,d_k)$,
\item $[c_k,d_k) \rightarrow [c_k,(\alpha^2+\alpha \beta+\beta^2 c_k) d_k)$,
\end{enumerate}
for some $k \in \set{1,\dots,n}$ and $\alpha,\beta \in F$.
\end{thm}

\begin{proof}
Let $A$ be the tensor product under discussion.
Let $\mathcal{Y}$ be the set of all tensor products of quaternion algebras isomorphic to $A$, and $\mathcal{X}$ be the set of all standard quadratic forms $q \in I_q^2 F$ with nonzero second coefficients in all the summands such that $C(q) \cong M_2(A)$.
We define a map $\mathcal{A} : \mathcal{X} \rightarrow \mathcal{Y}$ by
$$\mathcal{A}([a_1,b_1] \perp \dots \perp [a_{n+1},b_{n+1}])=[a_1 b_1,b_{n+1} b_1) \otimes \dots \otimes [a_n b_n,b_{n+1} b_n).$$

The map $\mathcal{A}$ is surjective because for each $[c_1,d_1) \otimes \dots \otimes [c_n,d_n)$, $$\mathcal{A}([c_1 d_1^{-1},d_1] \perp \dots \perp [c_n d_n^{-1},d_n] \perp [c_1+\dots+c_n,1])=[c_1,d_1) \otimes \dots \otimes [c_n,d_n).$$

Let $[c_1,d_1) \otimes \dots \otimes [c_n,d_n),[c_1',d_1') \otimes \dots \otimes [c_n',d_n') \in \mathcal{Y}$ and let \\$[a_1,b_1] \perp \dots \perp [a_{n+1},b_{n+1}],[a_1',b_1'] \perp \dots \perp [a_{n+1}',b_{n+1}'] \in \mathcal{X}$ such that $$\mathcal{A}([a_1,b_1] \perp \dots \perp [a_{n+1},b_{n+1}])=[c_1,d_1) \otimes \dots [c_n,d_n)$$ $$\mathcal{A}([a_1',b_1'] \perp \dots \perp [a_{n+1}',b_{n+1}'])=[c_1',d_1') \otimes \dots [c_n',d_n').$$

If $I_q^3 F=0$ then $[a_1,b_1] \perp \dots \perp [a_{n+1},b_{n+1}]$ and $[a_1',b_1'] \perp \dots \perp [a_{n+1}',b_{n+1}']$ are isometric.
If $n=2$ then they are similar (see \cite{Jacobson1983} and \cite{MammoneShapiro}). In this case $[a_1',b_1'] \perp [a_2',b_2'] \perp [a_3',b_3'] \simeq \beta ([a_1,b_1] \perp [a_2,b_2] \perp [a_3,b_3]) \simeq [\beta a_1,\beta^{-1} b_1] \perp [\beta a_2,\beta^{-1} b_2] \perp [\beta a_3,\beta^{-1} b_3]$ for some $\beta \in F^\times$. Since  $\mathcal{A}([\beta a_1,\beta^{-1} b_1] \perp [\beta a_2,\beta^{-1} b_2] \perp [\beta a_3,\beta^{-1} b_3])$ is connected to $[c_1,d_1) \otimes [c_2,d_2) \otimes [c_3,d_3)$ by several steps of Type 4, we can assume without loss of generality that $[a_1,b_1] \perp \dots \perp [a_{n+1},b_{n+1}]$ and $[a_1',b_1'] \perp \dots \perp [a_{n+1}',b_{n+1}']$ are isometric. Therefore they are connected by the chain described in Theorem \ref{qfchain} and Remark \ref{nonzero}. It is enough to consider the case where they are connected by exactly one step, and the general case is obtained by induction.

Assume $[a_1,b_1] \perp \dots \perp [a_{n+1},b_{n+1}]$ and $[a_1',b_1'] \perp \dots \perp [a_{n+1}',b_{n+1}']$ are connected by a step of Type \textbf{A}. If $k \leq n-1$ then $c_i'=c_i$ and $d_i'=d_i$ for all $i \neq k,k+1$, and $c_k'=(a_k+b_{k+1}) b_k=c_k+(\frac{1}{b_{n+1}})^2 d_k d_{k+1}$, $d_k'=d_k$, $c_{k+1}'=(a_{k+1}+b_k) b_{k+1}=c_{k+1}+(\frac{1}{b_{n+1}})^2 d_k d_{k+1}$ and $d_{k+1}'=d_{k+1}$. In this case, $[c_1,d_1) \otimes \dots \otimes [c_n,d_n)$ and $[c_1',d_1') \otimes \dots \otimes [c_n',d_n')$ are connected by a step of Type 1.
If $k=n$ then they are connected by a step of Type 3.

Assume $[a_1,b_1] \perp \dots \perp [a_{n+1},b_{n+1}]$ and $[a_1',b_1'] \perp \dots \perp [a_{n+1}',b_{n+1}']$ are connected by a step of Type \textbf{B}. If $k \leq n$ then $[c_1,d_1) \otimes \dots \otimes [c_n,d_n)$ and $[c_1',d_1') \otimes \dots \otimes [c_n',d_n')$ can disagree only in the $k$th factor. In this case it reduces to the known chain equivalence for quaternion algebras, and so they are connected by steps of Types 3 and 4. If $k=n+1$ then they are connected by several steps of Type 4.

Assume $[a_1,b_1] \perp \dots \perp [a_{n+1},b_{n+1}]$ and $[a_1',b_1'] \perp \dots \perp [a_{n+1}',b_{n+1}']$ are connected by a step of Type \textbf{C}.
If $k \leq n$ then $[c_1,d_1) \otimes \dots \otimes [c_n,d_n)$ and $[c_1',d_1') \otimes \dots \otimes [c_n',d_n')$ can disagree only in the $k$th factor.
If $k=n+1$ then $[c_1,d_1) \otimes \dots \otimes [c_n,d_n)$ and $[c_1',d_1') \otimes \dots \otimes [c_n',d_n')$ agree coefficient-wise.

Assume $[a_1,b_1] \perp \dots \perp [a_{n+1},b_{n+1}]$ and $[a_1',b_1'] \perp \dots \perp [a_{n+1}',b_{n+1}']$ are connected by a step of Type \textbf{D}.
If $k \leq n$ then $[c_1,d_1) \otimes \dots \otimes [c_n,d_n)$ and $[c_1',d_1') \otimes \dots \otimes [c_n',d_n')$ can disagree only in the $k$th factor.
If $k=n+1$ then $c_i'=c_i$ and $d_i'=(b_{n+1}+\beta^2 a_{n+1}+\beta) b_i=d_i (1+(\frac{\beta}{b_{n+1}})^2 a_{n+1} b_{n+1}+\frac{\beta}{b_{n+1}})=d_i (1+(\frac{\beta}{b_{n+1}})^2 \sum_{r=1}^n c_r+\frac{\beta}{b_{n+1}})$ for all $i$.
Write $\gamma=\frac{\beta}{a_{n+1}}$ and we get $d_i'=d_i (1+\gamma^2 \sum_{r=1}^n c_r+\gamma)$.
In this case $[c_1,d_1) \otimes \dots \otimes [c_n,d_n)$ and $[c_1',d_1') \otimes \dots \otimes [c_n',d_n')$ are connected by a step of Type 2.
\end{proof}

\section{Tensor Products of Quaternion Algebras in Characteristic not $2$}\label{charnot2}

The following theorem recalls Witt's well-known chain equivalence for quadratic forms (see \cite{EKM} for reference):
\begin{thm}\label{qfcn2cl}
Over a field of characteristic not $2$, every two isomorphic diagonal nondegenerate quadratic forms are connected by a finite number of steps, where each step is either
\begin{itemize}
\item[\textbf{A}] $\langle a_k,a_{k+1}\rangle \rightarrow \langle(1+a_{k+1} a_k) a_k,(1+a_{k+1} a_k) a_{k+1}\rangle$
\end{itemize}
for some $k \in \set{1,\dots,n-1}$ and $\beta \in F^\times$, or
\begin{itemize}
\item[\textbf{B}] $\langle a_k\rangle \rightarrow \langle\beta^2 a_k\rangle$
\end{itemize}
for some $k \in \set{1,\dots,n}$ and $\beta \in F^\times$.
\end{thm}

Let $F$ be a field of characteristic not 2.

\begin{thm}
If $I_q^3 F=0$ then every two isomorphic tensor products of $n$ quaternion algebras are connected by a finite number of steps, where each step is either
\begin{enumerate}
\item $(c_k,d_k) \otimes (c_{k+1},d_{k+1}) \rightarrow (c_k,(\beta^2-c_k c_{k+1}) d_k) \otimes (c_{k+1},(\beta^2-c_k c_{k+1}) d_{k+1})$
\end{enumerate}
for some $k \in \set{1,\dots,n-1}$ and $\beta \in F$, or one of the following types:
\begin{enumerate}\setcounter{enumi}{1}
\item $(c_k,d_k) \rightarrow ((\alpha^2-\beta^2 d_k) c_k,d_k)$,
\item $(c_k,d_k) \rightarrow (c_k,(\alpha^2-\beta^2 c_k) d_k)$,
\end{enumerate}
for some $k \in \set{1,\dots,n}$ and $\alpha,\beta \in F$.
\end{thm}

\begin{proof}
Let $A$ be the tensor product of quaternion algebras under discussion.
Let $\mathcal{Y}$ be the set of all tensor products of quaternion algebras isomorphic to $A$.
Let $\mathcal{X}$ be the set of all diagonal nondegenerate quadratic forms $q \in I_q^2 F$ such that $C(q) \cong M_2(A)$.

We define a map $\mathcal{A} : \mathcal{X} \rightarrow \mathcal{Y}$ by
\begin{eqnarray*}
\mathcal{A}(\langle a_1,b_1,\dots,a_{n+1},b_{n+1})\rangle&=&(\gamma_1 a_1,\gamma_1 b_1) \otimes (\gamma_2 a_2,\gamma_2 b_2) \otimes \dots \otimes (\gamma_n a_n,\gamma_n b_n).\\
\gamma_k&=&(-1)^k (\prod_{i=1}^{k-1} a_i b_i) b_{n+1}.
\end{eqnarray*}

The map $\mathcal{A}$ is surjective because for every presentation $(c_1,d_1) \otimes \dots \otimes (c_n,d_n)$, $\mathcal{A}(f)=(c_1,d_1) \otimes \dots \otimes (c_n,d_n)$ where
\begin{eqnarray*}
f&=&\langle \delta_1 c_1,\delta_1 d_1,\delta_2 c_2,\delta_2 d_2\dots,\delta_n c_n,\delta_n d_n,\delta_{n+1},-1\rangle\\
\delta_k&=&(-1)^{k-1} \prod_{i=1}^{k-1} (c_i d_i)^{(-1)^{i+k}}.
\end{eqnarray*}

Let $(c_1,d_1) \otimes \dots \otimes (c_n,d_n),(c_1',d_1') \otimes \dots \otimes (c_n',d_n') \in \mathcal{Y}$ and let \\$\langle a_1,b_1,\dots,a_{n+1},-1\rangle,\langle a_1',b_1',\dots,a_{n+1}',-1\rangle \in \mathcal{X}$ such that $$\mathcal{A}(\langle a_1,b_1,\dots,a_{n+1},b_{n+1}\rangle)=(c_1,d_1) \otimes \dots \otimes (c_n,d_n)$$  $$\mathcal{A}(\langle a_1',b_1',\dots,a_{n+1}',b_{n+1}'\rangle)=(c_1',d_1') \otimes \dots \otimes (c_n',d_n').$$

Since $I_q^3 F=0$, $\langle a_1,b_1,\dots,a_{n+1},-1\rangle$ and $\langle a_1',b_1',\dots,a_{n+1}',-1\rangle $ are isometric. Therefore, because of the Witt cancellation theorem, they are connected by the chain described in Theorem \ref{qfcn2cl} such that the last entry remains $-1$ throughout.
It is enough to consider the case where they are connected by exactly one step, and the general case is obtained by induction.

If $\langle a_1,b_1,\dots,a_{n+1},-1\rangle$ and $\langle a_1',b_1',\dots,a_{n+1}',-1\rangle $ are connected by a step of Type \textbf{B} then $(c_1,d_1) \otimes \dots \otimes (c_n,d_n)$ and $(c_1',d_1') \otimes \dots \otimes (c_n',d_n')$ are connected by several steps of Types 2 and 3.

Assume $\langle a_1,b_1,\dots,a_{n+1},-1\rangle$ and $\langle a_1',b_1',\dots,a_{n+1}',-1\rangle $ are connected by a step of Type \textbf{A}.
This step acts on either $\langle a_k,b_k\rangle$ or $\langle b_k,a_{k+1} \rangle$ for some $1 \leq k \leq n$.
If it acts on $\langle a_k,b_k\rangle$ then $(c_1,d_1) \otimes \dots \otimes (c_n,d_n)$ and $(c_1',d_1') \otimes \dots \otimes (c_n',d_n')$ can disagree in only one factor, and therefore they are connected by several steps of Types 2 and 3.

If it acts on $\langle b_k,a_{k+1} \rangle$ and $k \leq n-1$ then $c_i'=c_i$ and $d_i'=d_i$ for $i \leq k$, $c_k'=c_k$, $d_k'=(1+b_k a_{k+1}) d_k$, $c_{k+1}'=(1+b_k a_{k+1})^2 c_{k+1}$, $d_{k+1}'=(1+b_k a_{k+1}) d_{k+1}$ and $c_i'=(1+b_k a_{k+1})^2 c_i$ and $d_i'=(1+b_k a_{k+1})^2 d_i$ for $i \geq k+2$. $b_k a_{k+1} \equiv -c_k c_{k+1} \pmod{(F^\times)^2}$, and therefore $1+b_k a_{k+1} \equiv \beta^2-c_k c_{k+1} \pmod{(F^\times)^2}$ for some $\beta \in F^\times$. Consequently $(c_1,d_1) \otimes \dots \otimes (c_n,d_n)$ and $(c_1',d_1') \otimes \dots \otimes (c_n',d_n')$ are connected by one step of Type 1 and several steps of Type 3.

If it acts on $\langle b_n,a_{n+1}\rangle$ then $(c_1,d_1) \otimes \dots \otimes (c_n,d_n)$ and $(c_1',d_1') \otimes \dots \otimes (c_n',d_n')$ are connected by one step of Type 3.
\end{proof}

\section*{Acknowledgements}
I would like to thank Jean-Pierre Tignol, Uzi Vishne and the anonymous referee, whose comments improved the quality of the paper considerably.

\section*{Bibliography}
\bibliographystyle{amsalpha}
\bibliography{bibfile}
\end{document}